\documentclass[11pt]{amsart}

\usepackage{amsmath}
\usepackage{amssymb}
\usepackage[margin=1in]{geometry}
\usepackage{xcolor}
\usepackage{graphicx}
\usepackage{microtype}
\usepackage[msc-links]{amsrefs}
\usepackage{hyperref}
\usepackage{enumitem}
\hypersetup{
    colorlinks=false,
    linkcolor= Red,
    filecolor=BrickRed,      
    urlcolor=MidnightBlue,
    pdfpagemode=FullScreen,
    }
\usepackage{amsthm}

\theoremstyle{plain}
\newtheorem{theorem}{Theorem}[section]
\newtheorem{corollary}[theorem]{Corollary}
\newtheorem{lemma}[theorem]{Lemma}

\theoremstyle{definition}
\newtheorem{definition}{Definition}[section]

\theoremstyle{remark}

\newcommand{\A}{\mathbb A}
\renewcommand{\P}{\mathbb P}
\newcommand{\C}{\mathbb C}

\newcommand{\Q}{\mathbb Q}
\newcommand{\Z}{\mathbb Z}

\newcommand{\abs}[1]{\left\lvert #1 \right\rvert}

\newcommand{\Mod}[4]{\mathcal{M}_{#1}^{#2}[\mathcal{#3}_{#4}]}
\newcommand{\Aut}[1]{\mathrm{Aut}(#1)}
\newcommand{\restr}[1]{|_{#1}}

\makeatletter
\def\blfootnote{\gdef\@thefnmark{}\@footnotetext}
\makeatother

\DeclareMathOperator{\Bl}{Bl}
\DeclareMathOperator{\Pic}{Pic}

\title{Moduli Spaces of Degree Two Rational Maps with Portraits Up to Six Points}
\author{Louis Diaz}
\address{Department of Mathematics, The Pennsylvania State University, University Park, Pa 16802}
\email{lfd5286@psu.edu}
\date{\today}

\begin{document}

    \begin{abstract}
        We consider the moduli space, $\Mod{d}{N}{P}{}$, of degree $d$ rational self maps of $\P^N$ with prescribed pre-periodic structure $\mathcal{P}$ which were introduced by Doyle and Silverman. 
        It was shown in \cite{Blanc_2015} that, $\Mod{2}{1}{P}{6}$, the moduli space of degree two rational maps with a $6$-cycle is a surface of general type. 
        Here we compute the Kodaira dimension of all the moduli spaces with up to six pre-periodic points and show that $\kappa = -\infty, 0, 1,2$ are all realized for some $\mathcal{P}$. 
    \end{abstract}

\maketitle

    \section{Introduction}

    A primary question in arithmetic dynamics concerns pre-periodic points of a polynomial map $f: \Q \to \Q$.  
    Let $f^{(n)}$ denote the $n$-th iterate of $f$ under composition. 
    We say that a point $p$ is \emph{periodic of order n} if $f^{(n)}(p) = p$ and if $f^{(i)}(p) \neq p$ for $0 < i < n$. 
    A point is \emph{pre-periodic} if there exists a non-negative integer $m$ such that the point $f^{(m)}(p)$ is periodic. 

    It was conjectured in \cite{flynn1995cyclesquadraticpolynomialsrational} that if $f$ is quadratic, then it has no rational periodic points of order $n>3$. 
    This is also known as Poonen's conjecture, due to the progress made in \cite{poonen1995completeclassificationrationalpreperiodic}.  
    There is evidence for the $6$-cycle case in \cite{flynn1995cyclesquadraticpolynomialsrational}.
    However, it was also shown in \cite{Blanc_2015} that if, instead, $f: \P^1 \to \P^1$ is a degree two rational map, there are several infinite families of maps and rational periodic points of order $6$. 

    The collection of rational maps and points $(f, p_1,...,p_n)$ that exhibit a prescribed dynamic structure naturally forms a moduli space. 
    The study of these moduli spaces is motivated by the Uniform Boundedness Conjecture posed by Morton and Silverman \cite{Morton1994RationalPP}.  
    The conjecture asserts that for every $D \geq 1, N \geq 1$, and $d \geq 2$, there is a constant $C(D,N,d)$ such that for all number fields $K/\Q$, satisfying $[K:\Q] \leq D$ and all endomorphisms of $\P^N(K)$ of degree $d$, the number of pre-periodic points of $f$ is bounded above by a constant, $C(D,N,d)$. 
    Even in the simplest case, $(D,N,d) = (1,1,2)$, the answer is not known. 
    
    Although this conjecture is rooted in number theory and arithmetic dynamics, we can understand it from the perspective of geometry. 
    These moduli spaces, denoted as $\Mod{d}{N}{P}{}$ where $d$ is the degree of the rational maps $f$, $N$ is the dimension of the projective space, and $\mathcal{P}$ is the prescribed dynamic structure, which were constructed in \cite{Doyle_2020}, naturally arise from this problem. 
    
    Silverman and Doyle's conjecture \cite[Conjecture 10.6]{Doyle_2020} asserts that if $N \geq 1, d \geq1$, then there is a constant $C(N,d)$ such that for all unweighted pre-periodic portraits $\mathcal{P}$ with more points than $C(N,d)$ and $\Mod{d}{N}{P}{} \neq \emptyset$, the variety $\Mod{d}{N}{P}{} \otimes_{\Z} \C$ is an irreducible variety of general type.   
    The moduli space structure identifies rational tuples on $\Mod{d}{N}{P}{}$ with rational points exhibiting the portrait structure. 
    Since varieties of general type are conjectured to have sparse rational points, a general type moduli space would indicate that the rational points exhibiting the portrait structure are not Zariski dense. 

    Suppose that $\mathcal{P}$ is a prescribed pre-periodic structure on six points and $f$ is a degree two rational map on $\P^1$ over $\Q$. 
    Then the moduli space of degree two rational maps that exhibit the prescribed structure can be realized as a surface in $\P^3$. 
    Blanc, Canci, and Elkies showed in \cite{Blanc_2015} that if $\mathcal{P}_6$ is a $6$-cycle, $\Mod{2}{1}{P}{6}$ is a surface of general type. 
    They also showed that if $\mathcal{P}_5$ is a $5$-cycle, then $\Mod{2}{1}{P}{5}$ is a rational surface. 

    In this article, we study the case where $\mathcal{P}$ has at most $6$ points and decomposes as a union of cycles and $f$ is a rational map of degree two on $\P^1$. 
    The remaining cases where $\mathcal{P}$ is a cycle are investigated in \cite{Blanc_2015}.     
    Milnor proved in \cite{milnor1993geometry} that the moduli space $\mathrm{Rat}_2^1 \mathbin{/\!/} \Aut{\P^1} \cong \C^2$ where the coordinates are given by the symmetric polynomials  on the multipliers of the fixed points of $f$. 
    Silverman later showed that $\mathrm{Rat}_d^1 \mathbin{/\!/} \Aut{\P^1}$ exists as a geometric quotient over $\Z$, which uses the coefficients of $f$ for the coordinates \cite{silverman1996spacerationalmapsp1}.  
    It was shown that \cites{Blanc_2015, manes2009modulispacesfamiliesrational} that the surfaces $\Mod{2}{1}{P}{n}$ are geometrically irreducible for $n > 1$ and that they are all rational for $1 \leq n \leq 5$. 
    The six point case is the first instance that the moduli spaces are not rational. 
    However, six points is not enough to guarantee that the moduli space is general type; there are many six-point portraits for which the moduli spaces yield other values for the Kodaira dimension. 
    If the Silverman and Doyle conjecture is true, we must consider more points. 
    
    If $n$ is a positive integer, then each partition $n = i_1 + \cdots i_r$ determines a pre-periodic structure $\mathcal{P}_{i_1, \cdots, i_r}$ on $n$ points, which is the union of $i_1,...,i_r$-cycles. 
    Note that every pre-periodic structure on $n$ points arises from a partition of $n$. 
    The strategy to classify these surfaces is to embed each one as a surface in $\A^3$. 
    From there, we either explicitly resolve the singularities, use sub-pre-periodic structures to express it as a covering space, or embed it into the space of six ordered points on $\P^1$. 
    We will not be focusing on the $3,4,5,6$-cycle, as those were dealt with in \cite{Blanc_2015}. 
    In section 3 we show that all pre-periodic structure on $3,4,5$ points yield moduli spaces which are rational surfaces or empty. 
    In section 4, we construct different models for the moduli spaces that we will use to classify spaces in the $n=6$ case. 
    In section 5, we prove the following: 

    \begin{theorem}
        Let $\mathcal{P}_{5,1}, \mathcal{P}_{4,1,1}, \mathcal{P}_{4,2}, \mathcal{P}_{3,3}, \mathcal{P}_{3,2,1}, \mathcal{P}_{3,1,1,1}$ be pre-periodic portraits on six points and let $\kappa$ denote the Kodaira dimension. 
        Then,
            \begin{itemize}
                \item $\kappa = -\infty$ for $\Mod{2}{1}{P}{3,1,1,1}$, 
                \item $\kappa = 0$ for $\Mod{2}{1}{P}{4,2}$, 
                \item $\kappa = 1$ for $\Mod{2}{1}{P}{3,2,1}, \Mod{2}{1}{P}{3,3}, \Mod{2}{1}{P}{4,1,1}$, 
                \item $\kappa = 2$ for $\Mod{2}{1}{P}{5,1}$. 
            \end{itemize}
    \end{theorem}
    \noindent
    We use this and the results of section 5 to show:
        \begin{corollary}
            Suppose that $\mathcal{P}_6$ or $ \mathcal{P}_{5,1}$ are subportraits of $\mathcal{P}$. 
            If $\Mod{2}{1}{P}{} \neq \emptyset$, then it is a surface of general type. 
        \end{corollary}
    Silverman and Doyle conjecture that there exists an $N$ such that for every portrait with more than $N$ points, $\Mod{2}{1}{P}{}$ is either empty or general type. 
    Theorem 1.1 shows that $n=6$ is not sufficient. 
    If we assume that the moduli space $\Mod{2}{1}{P}{k}$ is general type for all sufficiently large $k$, then we have the following: 
    \begin{corollary}
        Suppose that $\Mod{2}{1}{P}{k}$ is general type for all $k\geq 6$. 
        Let $\mathcal{P}$ be a pre-periodic portrait with $N$ points where $N \geq 59$. 
        Then, $\Mod{2}{1}{P}{}$ is general type. 
    \end{corollary}

    \section{Preliminaries}

    We make precise some definitions from the introduction and work over $\C$. 
    
    \begin{definition}
        An \emph{unweighted pre-periodic portrait} is a pair 
            \[
                \mathcal{P} = (\mathcal{V}, \Psi)
            \]
        such that: 
            \begin{itemize}
                \item $\mathcal{V}$ is a finite set 
                \item $\Psi: \mathcal{V} \to \mathcal{V}$ is a function. 
            \end{itemize}
    \end{definition}
    \noindent
    An \emph{(unweighted pre-periodic) subportrait} of $\mathcal{P}$ is a portrait $(\mathcal{W}, \widehat{\Psi})$ such that $\mathcal{W} \subseteq \mathcal{V}$ and $\Psi\restr{\mathcal{W}} = \widehat{\Psi}$. 
    Since we are only considering unweighted pre-periodic portraits, we will refer to them as portraits. 
    
    We now describe the structure of these moduli spaces. 
    Let $\mathrm{Rat}_d^1$ denote the variety that parametrizes the endomorphisms of degree $d$ on $\P^1$. 
    Each endomorphism of $\P^1$ can be expressed as $f([U:V]) = [f_1([U:V]): f_2([U:V])]$, where, 
        \begin{align*}
            f_1([U:V]) & = a_0U^d + a_1U^{d-1}V + \cdots + a_{d-1}UV^{d-1} + a_dV^d,  \\
            f_2([U:V]) & = a_{d+1}U^d + a_{d+2}U^{d-1}V + \cdots + a_{2d}UV^{d-1} + a_{2d+1}V^d. 
        \end{align*}
    Taking the ordering $[a_0: \cdots : a_{2d+1}]$, we see that $\mathrm{Rat}_d^1 \subset \P^{2d+1}$. 
    It is an open subvariety of dimension $2d+1$, given by $\mathrm{Res}(f) \neq 0$, which is a polynomial in the coefficients, \cite{JOUANOLOU1991117}. 
    These maps behave well under conjugation. 
    That is, given $\varphi \in \Aut{\P^1}$, the map $\varphi^{-1} f \varphi$ has the same degree and the same preperiodic structure of $f$. 

    \begin{definition}
        Let $n \geq 1$ be an integer and $\mathcal{P}$ be a portrait on $n$ points. 
        Then, the space of degree $d$ rational maps on $\P^1$  that exhibit the portrait structure is given by 
            \[
                \mathrm{End}_d^1[\mathcal{P}] = \left\{(f, p_1, \dots p_n) \mid f(p_i) = p_{\Psi(i)} \text{ and } p_i \neq p_j \text{ for }i \neq j \right\} \subset \mathrm{Rat}_d^1 \times (\P^1)^n
            \]
        The action by $\Aut{\P^1}$ on $\mathrm{End}_d^1[\mathcal{P}]$ is given by
            \[
                \varphi \ast (f, p_1, ..., p_n) = (\varphi^{-1}f\varphi, \varphi^{-1}(p_1),..., \varphi^{-1}(p_n))
            \]
        and the moduli space of degree $d$ rational maps on $\P^1$ is given by $\Mod{d}{1}{P}{} := \mathrm{End}_d^1[\mathcal{P}] \mathbin{/\!/} \Aut{\P^1}$ and exists as a geometric quotient scheme \cite{Doyle_2020}. 
    \end{definition}

    We will be focusing on degree two rational maps so we write $[a_0:a_1:a_2:a_3:a_4:a_5] \in \P^5$ for the map $f([U:V]) = [a_0U^2 + a_1UV + a_2V^2: a_3U^2 + a_4UV + a_5V^2]$. 
    The open subvariety, $\mathrm{Rat}_2^1 \subset \P^5$, is given by 
        \[
            0 \neq a_2^2a_3^2 + a_0^2a_5^2 - 2a_0a_2a_3a_5 - a_1a_2a_3a_4 - a_0a_1a_4a_5 + a_0a_2a_4^2 + a_1^2a_3a_5. 
        \]

    We make a convention about the ordering of a tuple $(f, p_1, ..., p_{i+j}) \in \Mod{2}{1}{P}{i,j}$. 
    The points $p_1, ..., p_i$ form the $i$-cycle and $p_{i+1}, ..., p_{i+j}$ form the $j$-cycle. 
    We follow a similar convention for portraits with more indices. 
    We also make a note of how many cycles a degree two map $f: \P^1 \to \P^1$ has.     
    We write $f^{(n)}([U:V]) = [f_1^{(n)}([U:V]): f_2^{(n)}([U:V])]$ where the polynomials $f_1^{(n)}, f_2^{(n)}$ are given by the recursion 
        \begin{align*}
            f_1^{(n)}([U:V]) &= f_1^{(n-1)}([f_1([U:V]) : f_2([U:V])]), \\ 
            f_2^{(n)}([U:V]) &= f_2^{(n-1)}([f_1([U:V]) : f_2([U:V])]).
        \end{align*}
    Let $\Phi_{n,f}([U:V]) = V f_2^{(n)}([U:V]) - U f_1^{(n)}([U:V])$. 
    If $p = [u:v] \in \P^1$ is a root of this polynomial, then $f^{(n)}(p) = p$. 
    It is not necessarily true that $p$ is periodic of order $n$.  
    To fix this issue, we consider the $n$-th dynatomic polynomial. 
    \begin{definition}
        The $n$-th dynatomic polynomial for $f$ is given by 
            \[
                \Phi_{n,f}^*([U:V]) = \prod\limits_{k \mid n} \left( \Phi_{k,f}([U:V]) \right)^{\mu(n/k)}. 
            \]
    \end{definition} \noindent
    If $p= [u:v]$ is a point of period $n$, then it is a root of the $n$-th dynatomic polynomial. 
    The degree of the $n$-th dynatomic polynomial tells us how many points of period $n$ to expect. 
    Thus, $f$ can have at most three fixed points, one $2$-cycle, two $3$-cycles, three $4$-cycles, and six $5$-cycles. 
    The discriminant of $\Phi_{n,f}^*([U:V])$ encodes whenever any of the points of period $n$ has multiplicity greater than $1$. 
    In general, $p$ is not a root of a $k$-th dynatomic polynomial for $1 \leq k < n$. 
    An example is provided by Manes in \cite{manes2009modulispacesfamiliesrational} with $[U:V] \longmapsto [-U^2 + V^2: UV]$ with $p = [1:0]$. 

    Since we will often require a specific choice of coordinates on a blow-up, we fix some notation here. 
    The blow up of $\A^n$ at the origin, $\Bl_0 \A^n \subset \A^n \times \P^{n-1}$, can be expressed with coordinates $(x_1,...,x_n) \times [X_0: \cdots X_{n-1}]$ defined by $x_iX_j = x_j X_i$ for $1 \leq i,j \leq n$. 
    If we want to blow up a point $y \in A^n$ that is not the origin, we apply a linear transformation $(x_1 - y_1, ..., x_n - y_n)$ to shift it to the origin, proceed as above, and apply the inverse linear transformation. 
    If $V \subset \A^n$ is a singular subvariety, then the charts $X_i = 1$ let us identify the infinitely near singularities. 
    We will show that the moduli spaces can be realized as singular surfaces in $\A^3$ in the $6$ point case. 
    These moduli spaces will often require iterated blow ups and we use the product notation to identify singularities. 
    
    For example, consider the quintic curve , $C : \{ y^2 = x^5\} \subset \A^2$. 
    It has a singularity at $(0,0)$. 
    If we blow up the singularity, the strict transform $C'$ is still singular. 
    Using $(x,y) \times [U:V] \in \A^2 \times \P^1$ for our coordinates, then in the $U=1$ chart, our strict transform is given by $C': \{V^2 = x^3\}$, a cuspidal cubic, which has a double point at $(x,V) = (0,0)$. 
    We have that $C$ has an infinitely near singularity at $(0,0) \times [1:0]$. 
    To resolve this singularity, we deal with the affine chart with coordinates $(x,V) \in \A^2$ and continue the blow up process as before.

    \section{Portraits With Less Than Six Points}

    In this section we show that if $\mathcal{P}$ is a portrait with $3,4$, or $5$ points, $\Mod{2}{1}{P}{}$ is a rational surface. 
    The proof for each moduli space is done by a case-by-case analysis. 
    We do, however, take advantage of the action by $\Aut{\P^1}$ and take three of the points to be $[0:1], [1:0]$, and $[1:1]$. 

        \begin{lemma}
            Let $\mathcal{P}$ be a portrait on three points. 
            Then, $\Mod{2}{1}{P}{}$ is isomorphic to an open subset of $\A^2$. 
        \end{lemma}
        \begin{proof}
            If $\mathcal{P}$ is a portrait on $3$ points, that is not a cycle, then $\mathcal{P} = \mathcal{P}_{2,1}$ or $ \mathcal{P}_{1,1,1}$. 
                \begin{enumerate} [label = (\roman*)]
                    \item 
                        We embed $\Mod{2}{1}{P}{2,1}$ in $\mathrm{Rat}_2^1 \subset \P^5$. 
                        It is given by the maps $f$ which satisfy \newline $f([0:1]) = [1:0], f([1:0]) = [0:1]$, and $f([1:1])= [1:1]$.
                        These equations then show that $f$ must be of the form
                            \[
                                f([U:V]) = [a_1UV + a_2V^2: a_3U^2 +(a_1 + a_2 - a_3)UV]. 
                            \]
                        That is, it is an open subset of $\P^2$, with coordinates $[a_1: a_2:a_3]$. 
                        The corresponding resultant of $f$ is given by $-a_3 a_2 (a_1 - a_3)(a_1 + a_2)$, so that $\Mod{2}{1}{P}{2,1}$ is the complement of three lines in $\A^2$. 
                    \item 
                        We embed $\Mod{2}{1}{P}{1,1,1}$ in $\mathrm{Rat}_2^1 \subset \P^5$. 
                        It is given by the maps $f$ which fix $[0:1],[1:0]$, and $[1:1]$. 
                        These equations then show that $f$ must be of the form
                            \[ 
                                f([U:V]) = [a_0U^2 + a_1UV: a_4UV +(a_0 + a_1 - a_4)V^2].
                            \]
                        It is an open subset of $\P^2$ with coordinates $[a_0: a_1: a_4]$.  
                        The corresponding resultant of $f$ is $a_0(a_1 + a_0)(a_0 - a_4 )(a_0 + a_1 - a_4)$. 
                        Thus, $\Mod{2}{1}{P}{1,1,1}$ is also the complement of three lines in $\A^2$. 
                \end{enumerate}
        \end{proof}

        \begin{lemma}
            Let $\mathcal{P}$ be a portrait on four points. 
            Then, either $\Mod{2}{1}{P}{}$ is empty or isomorphic to an open subset of $\A^2$. 
        \end{lemma}
        \begin{proof}
            If $\mathcal{P}$ is a portrait on four points, that is not a cycle, then $\mathcal{P} = \mathcal{P}_{3,1}, \mathcal{P}_{2,2}, \mathcal{P}_{2,1,1}, \mathcal{P}_{1,1,1,1}$. 
                \begin{enumerate} [label = (\roman*)] 
                    \item 
                        If $\mathcal{P} = \mathcal{P}_{2,2}$ or $\mathcal{P}_{1,1,1,1}$, then $\Mod{2}{1}{P}{} = \emptyset$. 
                    \item 
                        We embed $\Mod{2}{1}{P}{3,1} \subset \mathrm{Rat}_2^1 \times \A^1 \subset \P^5 \times \A^1$. 
                        It is given by $(f,x)$ where \newline $f([0:1]) = [1:0], f([1:0]) = [x:1], f([x:1]) = [0:1]$, and $f([1:1]) = [1:1]$.  
                        So, the map $f$ corresponds to the endomorphism 
                            \[
                                f([U:V]) = [a_3xU^2 + a_1UV - (a_3x^3 + a_1x)V^2: a_3U^2 - (a_3x^3 + a_1x - a_3x - a_1 + a_3)UV ].  
                            \]
                        It is an open subset of $\P^1 \times \A^1$, with coordinates $([a_1:a_3],x)$. 
                        The resultant of $f$ is given by $a_3x(x-1)(x^2 - x + 1)(a_3x^2 + a_1)(a_3x^2 + a_3x + a_1)(a_3x^2+a_3x + a_1 - a_3)$, so that $\Mod{2}{1}{P}{3,1}$ is the complement of six lines in $\A^2$, when taking the affine chart $a_3 =1$.  
                    \item 
                        We embed $\Mod{2}{1}{P}{2,1,1} \subset \mathrm{Rat}_2^1 \times \A^1 \subset \P^5 \times \A^1$. 
                        It is given by the pair $(f,x)$, where $f([1:1]) = [x:1], f([x:1]) = [1:1]$ and $[0:1], [1:0]$ are fixed. 
                        Then, the map $f$ can be expressed as 
                            \[
                                f([U:V]) = [a_0(x+1)U^2 + (a_0x^2 + a_0x - a_4x + a_0)UV: a_4(x+1)UV - x(a_0 + a_4)V^2]. 
                            \]
                        It is an open subset of $\P^1 \times \A^1$, with coordinates $([a_0:a_4],x)$. 
                        The resultant of $f$ is given by $-a_0x(a_0 + a_4)(x+1)^2(a_4x^2 - a_0x)$. 
                        Thus, $\Mod{2}{1}{P}{2,1,1}$ is the complement of four lines in $\A^2$, when taking the affine chart $a_0 = 1$.  
                \end{enumerate}
        \end{proof}
        
        \begin{lemma}
            Let $\mathcal{P}$ be a portrait on five points. 
            Then, either $\Mod{2}{1}{P}{}$ is empty or isomorphic to an open subset of $\A^2$. 
        \end{lemma}
        \begin{proof}
            If $\mathcal{P}$ is a portrait on five points, that is not a cycle, then $\mathcal{P} = \mathcal{P}_{4,1}, \mathcal{P}_{3,2}, \mathcal{P}_{3,1,1}, \mathcal{P}_{2,1,1,1}, \mathcal{P}_{2,2,1}, \mathcal{P}_{1,1,1,1,1}$. 
                \begin{enumerate} [label = (\roman*)]
                    \item 
                        If $\mathcal{P} = \mathcal{P}_{2,2,1}$ or $\mathcal{P}_{1,1,1,1,1}$, then $\Mod{2}{1}{P}{} = \emptyset$. 
                    \item 
                        We embed $\Mod{2}{1}{P}{4,1} \subset \mathrm{Rat}_2^1 \times \A^2 \subset \P^5 \times \A^2$. 
                        The points are given by $(f,x,y)$ where $f$ cycles $[0:1], [1:0], [x:1], [y:1]$ and fixes $[1:1]$. 
                        The map $f$ can then be expressed as 
                            \[
                                f([U:V]) = [a_0U^2 + a_1UV +a_2V^2: a_3U^2 + a_4UV]
                            \]
                        where 
                            \begin{align*}
                                a_0 &= x(xy^2 - xy + x - y), \\
                                a_1 &= -x(xy^3 + x^2 - 2xy - y^2 + y), \\ 
                                a_2 &=  xy(xy^2 + x^2 - 3xy + y), \\ 
                                a_3 &= xy^2 - xy + x - y,  \\ 
                                a_4 &= x^3y - 2x^2y^2 - x^3 + x^2y + xy^2 + x^2 - xy - x - y,  \\ 
                            \end{align*}
                        and the resultant of $f$ is $x^2y(x-y)(x-2y)(x-1)(y-1)(xy - x + 1)(x^2 - xy - x + 1)(xy^2 - xy + x - y)(xy^2 + x^2 - 3xy + y)$ so that it is isomorphic to an open affine subset of $\A^2$. 
                    \item 
                        We embed $\Mod{2}{1}{P}{3,2} \subset \mathrm{Rat}_2^1 \times \A^2 \subset \P^5 \times \A^2$. 
                        The points are given by $(f,x,y)$ where $f$ cycles $[0:1],[1:0],[1:1]$ and swaps $[x:1], [y:1]$. 
                        Then, the map $f$ can be expressed as 
                            \[
                                f([U:V]) = [a_0U^2 + a_1UV + a_2V^2: a_0U^2 + a_4UV] 
                            \]
                        where 
                            \begin{align*}
                                a_0 & = xy(x + y + 1), \\
                                a_1 & = - xy(xy + x + y + 1), \\ 
                                a_2 & = x^2y^2, \\ 
                                a_4 & = - (x^2y^3 + x^2y^2 + xy^2 + xy + x + y + 1), \\ 
                            \end{align*}
                        and the resultant of $f$ is $x^4y^4(x-y)(x+1)^2(y+1)^2(xy + x + 1)(xy + y+1)$, so that it is isomorphic to an open affine subset of $\A^2$. 
                    \item 
                        We embed $\Mod{2}{1}{P}{3,1,1} \subset \mathrm{Rat}_2^1 \times \A^2 \subset \P^5 \times \A^2$. 
                        The points are given by $(f,x,y)$ where $f$ cycles $[1:1],[x:1],[y:1]$ and fixes $[0:1],[1:0]$. 
                        Then, the map $f$ can be expressed as 
                            \[
                                f([U:V]) = [a_0U^2 + a_1UV : a_4UV + a_5V^2]
                            \]
                        where 
                            \begin{align*}   
                                a_0 &= x^2y^2 - x^2y + x^2 - xy^2 - xy + y^2, \\
                                a_1 &= x^3y - x^3 - x^2y^3 + xy^3 + xy - y^2, \\
                                a_4 &= x^3y - x^2y^2 - x^2 + x + y^3 - y^2, \\
                                a_5 &= -xy(x^2 - xy + y^2 - x - y + 1), \\
                            \end{align*}
                        and the resultant of $f$ is $x y (x - 1) (y - 1) (x - y) (x y - 1) (x^2 - y) (x - y^2) (x^2 - x y - x + y^2 - y + 1) (x^2 y^2 - x^2 y + x^2 - x y^2 - x y + y^2)$ so that it is isomorphic to an open affine subset of $\A^2$. 
                    \item 
                        We embed $\Mod{2}{1}{P}{2,1,1,1} \subset \mathrm{Rat}_2^1 \times \A^2 \subset \P^5 \times \A^2$. 
                        The points are given by $(f,x,y)$ where $f$ swaps $[0:1],[1:0]$ and fixes $[1:1],[x:1],[y:1]$. 
                        Then, the map $f$ can be expressed as 
                            \[
                              f([U:V]) = [-(x + y + xy)UV + xyV^2: U^2 - (x+y+1)UV]
                            \]
                        where the resultant of $f$ is $xy(x + 1)(y + 1)(x + y)$, so that it is isomorphic to an open affine subset of $\A^2$. 
                \end{enumerate}
        \end{proof}
    Thus, we have shown that for all portraits on $n=3,4,5$ points, the moduli space is a rational surface. 
    \begin{theorem}
        Let $\mathcal{P}$ be a portrait on $n=3,4,5$ points. 
        Then, either $\Mod{2}{1}{P}{}$ is empty or a smooth rational surface. 
    \end{theorem}

    \section{Different Models of the Moduli Spaces}

    We now consider the case when $\mathcal{P}$ is a portrait on $6$ points.
    In order to classify these moduli spaces, we take advantage of different models.  
    For six points, there are a total of eleven portraits. 
    Recall that we are not considering the case where $\mathcal{P}$ is a cycle. 
    We classify the remaining ten moduli spaces by first showing that four of them yield empty moduli spaces. 

    \begin{lemma}
        Let $\mathcal{P}$ be one of the following portraits, $\mathcal{P}_{2,2,2}, \mathcal{P}_{2,2,1,1}, \mathcal{P}_{2,1,1,1,1}, \mathcal{P}_{1,1,1,1,1,1}$. 
        Then $\Mod{2}{1}{P}{} = \emptyset$. 
    \end{lemma}
    \begin{proof}
        We know that there are no degree two rational maps with more than three fixed points or more than one $2$-cycle. 
        Therefore, $\Mod{2}{1}{P}{} = \emptyset$. 
    \end{proof}

    Before we construct a model of these moduli spaces, we establish some notation. 
    Each moduli space, $\Mod{2}{1}{P}{}$, must satisfy open conditions, namely, all the points must be distinct and $\mathrm{Res}(f) \neq 0$.  
    Let $S$ be the closure of $\Mod{2}{1}{P}{}$ in $\A^2$ for the $3,4,5$ point case. 
    The indices of the surface $S$ correspond to which portrait we are considering and use a bar to denote the projective closure. 
    For example, we write $\Mod{2}{1}{P}{4,1} \subset S_{4,1} = \A^2$ and $\bar{S}_{4,1}= \P^2$. 

    \begin{lemma}
        Let $\mathcal{P} = \mathcal{P}_{5,1},\mathcal{P}_{4,2},\mathcal{P}_{4,1,1},\mathcal{P}_{3,3}, \mathcal{P}_{3,2,1}, \mathcal{P}_{3,1,1,1}$. 
        Then, each $\Mod{2}{1}{P}{}$ is isomorphic to  an open subset of a hypersurface in $\A^3$. 
    \end{lemma}
    \begin{proof}
        Recall that after the $\Aut{\P^1}$ action, the points in $\Mod{2}{1}{P}{}$ are given by $(f, [0:1],[1:0],[1:1],[x:1],[y:1],[z:1])$ with the condition that $f(p_i) = p_{\Psi(i)}$. 
        This condition is not linear in the points but is linear in the coefficients of $f$. 
        
        We build the matrix equation $M\vec{a} = 0$ where $\vec{a} = [a_0:a_1:a_2:a_3:a_4:a_5]^T$, $a_i$ are the coefficients of $f$ from definition 2.2, and $M\vec{a}$ is derived from $f(p_i) = p_{\Psi(i)}$.  
        To illustrate, the following matrix equation is for $\Mod{2}{1}{P}{5,1}$
            \[
                \begin{bmatrix}
                    0 & 0 & 0 & 0 & 0 & -1 \\
                    1 & 0 & 0 & -1 & 0 & 0 \\
                    1 & 1 & 1 & -x & -x & -x \\
                    x^2 & x & 1 & -x^2y & -xy & -y \\
                    y^2 & y & 1 & 0 & 0 & 0 \\
                    z^2 & z & 1 & -z^3 & -z^2 & -z \\
                \end{bmatrix}
                \begin{pmatrix}
                    a_0 \\
                    a_1 \\
                    a_2 \\
                    a_3 \\
                    a_4 \\
                    a_5 \\
                \end{pmatrix} 
                = 
                \begin{pmatrix}
                    0 \\ 
                    0 \\ 
                    0 \\
                    0 \\
                    0 \\
                    0 \\
                \end{pmatrix}  .
            \]
        If this matrix is full rank, then we only have the trivial solution, $\vec{a} = \vec{0}$. 
        The zeros of $\det{(M)}$ determine when the matrix drops in rank, which also defines a sextic surface in $\A^3$. 
        For each portrait $\mathcal{P}$, let $\mathcal{S}$ be the corresponding (possibly reducible) closed surface in $\A^3$. 

        Any tuple $(f, [0:1],[1:0],[1:1],[x:1],[y:1],[z:1]) \in \Mod{2}{1}{P}{}$ will always satisfy the determinant equation. 
        However, it is not necessarily true that every point on $\mathcal{S}$ corresponds to a tuple in $\Mod{2}{1}{P}{}$. 
        We require that none of the points in the tuple are the same; this guarantees that the points in the cycle do not collide. 
        We also require that the resultant of $f$ be nonzero so that the degree of $f$ is still two, where the coefficients of $f$ are determined by $x,y,z$. 
        These are open conditions in $\A^3$, so that $\Mod{2}{1}{P}{}$ is an open subset $\mathcal{S}$. 
        Following the convention above, let $S$ be the closure of $\Mod{2}{1}{P}{}$ in $\A^3$, which is an irreducible component of $\mathcal{S}$. 
        This is equivalent to removing the planar condition(s) from $\mathcal{S}$.  
        Following the procedure, we have that 
            \begin{itemize} 
                \item $\Mod{2}{1}{P}{5,1}, \Mod{2}{1}{P}{3,3}$  are isomorphic to open subsets of a sextic surface in  $\A^3$. 
                \item $\Mod{2}{1}{P}{4,2}, \Mod{2}{1}{P}{4,1,1}, \Mod{2}{1}{P}{3,2,1}$ are isomorphic to open subsets of  a quintic surface in $\A^3$. 
                \item $\Mod{2}{1}{P}{3,1,1,1}$ is isomorphic to an open subset of a cubic surface in  $\A^3$. 
            \end{itemize}
    \end{proof}

    This is not the only model for these moduli spaces. 
    We can also exhibit them as covering spaces via projection maps. 
    The projection maps are useful because they arise naturally from subportraits. 
    It is always possible to remove a subportrait condition; it is possible to recover some of the information after removing a subportrait condition. 
    That is, there is an explicit description to find the points in the fiber. 
    The fibers of the projection map depends heavily on how many distinct fixed points/cycles a degree two rational map can have. 

    \begin{lemma}
        Consider the portrait $\mathcal{P}_{n,1}$, where $n \geq 5$. 
        Then, there is a map from $S_{n,1} \to S_n$, which is generically $3$-to-$1$, where $S_{n,1}, S_n$ are defined by the cycle conditions shown in \cite{Blanc_2015}, and the branch locus is given by the discriminant of $\Phi_{1,f}^*([U:1])$. 
    \end{lemma}
    \begin{proof}
        Suppose that $n \geq 5$. 
        By \cite{Blanc_2015}, we have that $\Mod{2}{1}{P}{n} \subset \A^{n-3}$ is an open subvariety of $S_n$. 
        We also have that the fixed point equation defines a hypersurface, $H \subset\A^{n-2}$. 
        Then, the surface $S_{n,1} = H \cap S_n \subset \A^{n-3}$, where we are considering $S_n$ in $\A^{n-2}$ via $(x_1, \dots,x_{n-3}) \hookrightarrow (x_1, \dots, x_{n-3}, 0)$. 
        
        Let $(f, [0:1], [1:0], [1:1], p_4, \dots, p_n, \tilde{p}) \in \Mod{2}{1}{P}{n,1}$. 
        There is a point in $S_{n,1} \subset \A^{n-2}$ corresponding to this tuple 
        Then, the map, 
            \begin{align*}
                \Mod{2}{1}{P}{n,1} &\longrightarrow \Mod{2}{1}{P}{n} \\ 
                 (f, [0:1], [1:0], [1:1], p_4, \dots, p_n, \tilde{p}) &\longmapsto (f, [0:1], [1:0], [1:1], p_4, \dots, p_n)
            \end{align*}
        is the restriction of the projection map $S_{n,1} \to S_n$, $(x_1, \dots, x_{n-3}, x_{n-2}) \longmapsto (x_1, \dots, x_{n-3})$. 

        This projection map is ramified along a curve in $\A^{n-2}$ and has a corresponding branch locus in $\A^{n-3}$. 
        Since any degree two rational map has three fixed points, counting multiplicity, we have that this is generically a $3$-to-$1$ map. 
        The branch locus is given by the points in $\A^{n-3}$ whose fibers have less than three elements. 
        This occurs at the zeros of the discriminant of $\Phi_{1,f}^*([U:1])$ where taking the affine chart $V=1$ poses no problems. 
        Since all the coefficients of $f$ are given in terms of $x_1,x_2,x_{n-3}$, the branch locus is the intersection of the hypersurface defined by the discriminant and the surface $S_n \subset \A^{n-3}$. 
    \end{proof}
    
    \begin{lemma}
        Consider the portrait $\mathcal{P}_{n,m}$, with $n \geq 5$ and $m>1$.  
        Then, there is a natural projection map $\Mod{2}{1}{P}{n,m} \to \Mod{2}{1}{P}{n}$ where the branch locus is given by the discriminant of $\Phi_{n,f}^*([U:1])$. 
    \end{lemma}
    \begin{proof}
        We follow a similar format but now, we must keep track of how many distinct $m$-cycles a rational map of degree two can have. 
        Since $n \geq 5$, we can solve for the coefficients of $f$ in terms of the $n$-cycle. 
        Thus, $S_n \subset \A^{n-3}$. 
        We consider $S_n \subset \A^{n + m - 3}$ via the standard inclusion. 
        Then, each condition from the $m$-cycle defines a hypersurface. 
        Let $H_i$ be the hypersurface derived from $f(p_i) = p_{\Psi(i)}$ for $n < i \leq n+m$. 
        Then, we have $m$ hypersurfaces in $\A^{n+m-3}$. 
        Intersecting $m$ hypersurfaces along with $S_n \subset \A^{n+m-3}$, we have the surface $S_{n,m} = S_n \cap \bigcap_{i=n+1}^{n+m} H_i$. 
        
        Then, the map 
            \begin{align*}
                \Mod{2}{1}{P}{n,m} &\longrightarrow \Mod{2}{1}{P}{n} \\ 
                (f, [0:1], [1:0], [1:1], p_4, \dots, p_{n+m}) &\longmapsto (f, [0:1], [1:0], [1:1], p_4, \dots, p_n)
            \end{align*}
        is the restriction of the projection map $S_{n,m} \to S_n$, $(x_1, \dots, x_{n+m-3}) \longmapsto (x_1, \dots, x_{n-3})$. 

        In the fixed point case, we had three possible choices because there were three solutions, counting multiplicity, to the fixed point equation. 
        In the $m$-cycle case, it is possible to have more than one $m$-cycle. 
        If $f$ only has one $m$-cycle, then up to reordering, we have $m$ different elements in the fiber. 
        If $f$ has $k$ distinct $m$-cycles, then we have $km$ different elements in the fiber.  

        This map is ramified when the cycle degenerates. 
        The discriminant of $\Phi_{m,f}^*([U:1])$ identifies when the solutions to $f^{(m)}(p)=p$ collide. 
        It also defines a hypersurface in $\A^{n-3}$, in terms of $x_1,x_2,x_{n-3}$. 
        Thus, intersecting the hypersurface defined by the discriminant of $\Phi_{n,f}^*([U:1])$ with the surface $S_n$ gives the branch locus.  
    \end{proof}

    \section{Classifying the Surfaces}

    We now start to classify the surfaces for portraits with six points. 
    Using the tools of birational geometry, we can resolve the singularities if we realize it as a surface in $\A^3$, use the projection maps and the branch loci, or find a new model in the space of six ordered points, $P_1^6$. 


    \begin{lemma}
        The surface $\mathcal{M}_2^1[\mathcal{P}_{3,1,1,1}]$ is rational. 
        Further, $S_{3,1,1,1}$ is a degree three weak del Pezzo surface. 
    \end{lemma}
    \begin{proof}
        We use the model for this surface in $\mathbb{A}^3$. 
        The surface $S_{3,1,1,1}$ is the cubic surface defined by 
            \[
                0 = xyz - xy - xz - yz + 1. 
            \]
        Taking the projective closure, with coordinates $[x:y:z:w]$, we see that $\overline{S}_{3,1,1,1}$ only has the following double points
            \begin{align*}
                p_1 &= [1:0:0:0], & p_2 &= [0:1:0:0], & p_3 &= [0:0:1:0]. 
            \end{align*}
        Denoting by $\pi: \Bl_3 \P^3 \to \P^3$ the blow up of the three points, the strict transform $\overline{S}_{3,1,1,1}'$ of $\overline{S}_{3,1,1,1}$ is a smooth surface. 
        Let $E_i$ be the exceptional divisor from blowing up $p_i$ according to the order above and $\widetilde{H}$ be the pullback of a hyperplane in $\P^3$. 
        The canonical class of $\Bl_3 \P^3$ is $K_{\Bl_3 \P^3} = -4\widetilde{H} + 2 E_1 + 2E_2 + 2E_3$ and $\overline{S}_{3,1,1,1}'$ is equivalent to $3\widetilde{H} - 2E_1 - 2E_2 - 2E_3$. 
        The adjunction formula gives $K_{\overline{S}_{3,1,1,1}'} = -\widetilde{H}\restr{\overline{S}_{3,1,1,1}'}$. 

        We see that $K_{\overline{S}_{3,1,1,1}'}^2 = 3$, so that $\overline{S}_{3,1,1,1}'$ is a degree three weak del Pezzo surface with $-K_{\overline{S}_{3,1,1,1}'}$ big and nef. 
        Thus, $\Mod{2}{1}{P}{3,1,1,1}$ has Kodaira dimension $-\infty$. 
    \end{proof} 


    \begin{lemma}
        The surface $\Mod{2}{1}{P}{4,2}$ has Kodaira dimension $0$. 
    \end{lemma}
    \begin{proof}
        Recall that a degree two rational map on $\P^1$has at most one $2$-cycle. 
        Therefore, the projection map, $p: \Mod{2}{1}{P}{4,2} \to \Mod{2}{1}{P}{4}$, is a double cover. 
        It was shown in \cite{Blanc_2015} that $S_4 \cong \A^2$, which only depends on $(a_1,x)$. 
        We extend the projection map to $p: S_{4,2} \to S_4$ and investigate the branch locus. 

        The map $p$ has a branch locus given by the zeros of the discriminant of $\Phi_{2,f}^*([U:1])$. 
        The coefficients of $f$ are given by 
            \begin{align*}
                a_0 &= -a_1x - x^2,  \\ 
                a_2 &= 1, \\ 
                a_3 &= -a_1x - x^2, \\
                a_4 &= -a_1x^2 - x^3 + 2a_1x + x + x^2, \\ 
                a_5 &= 0. 
            \end{align*}
        The $2$-nd dynatomic polynomial is given by $AU^2 + BU + C$, where 
        \begin{align*}
            A &= -x(x-1)(a_1 + x) (a_1 + x + 1), \\
            B &= a_1x^2 + a_1^2x - x - a_1^2 - 2a_1, \\ 
            C &= a_1 + x. 
        \end{align*}
        The branch locus, $\mathcal{B}$, is a sextic curve in $\A^2$. 
        Taking the projective closure in $\P^2$, with coordinates $[a_1:x:w]$, we see that that $\overline{\mathcal{B}}$ is singular at the following double points
            \begin{align*}
                p_1 &= [1:-1:0], & p_2 &= [1:0:0], & p_3 &= [1:-1:-1], & p_4 &= [0:0:1].  
            \end{align*}

        Let $\pi_1: \Bl_4\P^2 \to \P^2$ be the blow up of these four points. 
        The strict transform of $\overline{\mathcal{B}}$ is still not smooth; it is singular at the following infinitely near double points
            \begin{align*}
                q_1 &= p_1 \times [0:1], & q_2 &= p_2 \times [1:1]. 
            \end{align*}
        Letting $\pi_2: \Bl_{4,2} \P^2 \to \Bl_4 \P^2$, be the blow up of the infinitely near points, the resulting strict transform $\overline{\mathcal{B}}'$ of $\overline{\mathcal{B}}$ is now smooth. 

        We denote $E_3,E_4$ to be the exceptional divisors obtained by blowing up $p_3,p_4$, $\widetilde{E}_i$ be the strict transform of the exceptional divisors after blowing up the infinitely near points, and $F_i$ from blowing up $q_i$ and $\widetilde{H}$ be the pullback of a general line in $\P^2$. 
        Letting $\pi = \pi_2 \circ \pi_1$ and using the ramification formula, we have 
            \begin{align*}
                K_{\Bl_{4,2} \P^2} &= -3\widetilde{H} + \widetilde{E}_1 + \widetilde{E}_2 + E_3 + E_4 + 2F_1 + 2F_2, \\ 
                \overline{\mathcal{B}}' &= 6\widetilde{H} - 2\widetilde{E}_1 - 2\widetilde{E}_2 -2E_3 - 2E_4 - 4F_1 - 4F_2. 
            \end{align*}
        By applying the Riemann-Hurwitz formula, and writing $\mathcal{R}'$ for the strict transform of the resolved ramification locus, we have that
            \begin{align*}
                K_{\overline{S}_{4,2}'} &= \pi^* \left( K_{\Bl_{4,2}\P^2} \right) + \mathcal{R}' \\
                                        &= \pi^*\left( K_{\Bl_{4,2}\P^2} + \frac{\overline{\mathcal{B}}'}{2} \right) \\ 
                                        &= 0, 
            \end{align*}
        where $\overline{S}_{4,2}'$ denotes the resolved projective surface from blowing up $\overline{\mathcal{B}}$. 
        The canonical divisor is trivial; we now look to the irregularity of $\overline{S}_{4,2}$. 
        The space $\overline{S}_{4,2}'$ is a double cover of $\Bl_{4,2}\P^2$ branched along $\overline{\mathcal{B}}'$. 
        Let $\mathcal{L}^{\otimes 2} = \mathcal{O}_{\Bl_{4,2}\P^2}{\overline{\mathcal{B}}'}$. 
        Thus, $\pi_*\mathcal{O}_{\overline{S}_{4,2}'} \cong \mathcal{O}_{\Bl_{4,2}\P^2} \oplus \mathcal{L}^{-1}$ \cite{barth2003compact}. 
        Therefore, 
            \[
                H^1(\overline{S}_{4,2}', \mathcal{O}_{\overline{S}_{4,2}'}) \cong H^1(\Bl_{4,2} \P^2, \mathcal{O}_{\Bl_{4,2} \P^2}) \oplus H^1(\Bl_{4,2} \P^2, \mathcal{L}^{-1}), 
            \]
        So that $h^1(\overline{S}_{4,2}', \mathcal{O}_{\overline{S}_{4,2}'}) = 0$. 
        The irregularity is $0$, $\overline{S}_{4,2}'$ is a $K3$ surface and $\Mod{2}{1}{P}{4,2}$ has Kodaira dimension $0$.  
    \end{proof}

    
    \begin{lemma}
        The surface $\Mod{2}{1}{P}{3,2,1}$ has Kodaira dimension 1. 
    \end{lemma}
    \begin{proof}
        The surface $S_{3,2,1} \subset \A^3$ is given by the following equation
            \[
                0 = -x^2y^2z + xyz^3 + x^2y^2 + x^2yz + xy^2z - 3xyz^2 - x^2y - xy^2 + xz^2 + yz^2 + xy - z^2. 
            \] 
        Taking its projective closure, with coordinates $[x:y:z:w]$, the surface $\overline{S}_{3,2,1}$ is singular at the following nine points, 
            \begin{align*}
                p_1 &= [1+i\sqrt{3}:1+i\sqrt{3}:1+i\sqrt{3}:2], & p_2 &= [1-i\sqrt{3}:1-i\sqrt{3}:1-i\sqrt{3}:2], \\ 
                p_3 &= [0:0:0:1], & p_4 &= [0:1:0:1],  \\ 
                p_5 &= [1:0:0:1], & p_6 &= [1:1:1:1], \\
                p_7 &= [0:0:1:0], & p_8 &= [1:0:0:0],   &                \\
                p_9 &= [0:1:0:0], 
            \end{align*}
        where $p_i, 1\leq i \leq 7$ are double points and $p_8, p_9$ are triple points. 
        Let $\pi_1: \Bl_9 \P^3 \to \P^3$ be the blow up of the nine points. 
        After blowing up $p_8,p_9$, a direct calculation shows that the surface is still singular at the following infinitely near double points
            \begin{align*}
                q_{8,1} &= p_8 \times [1:1:1], & q_{8,2} = p_8 \times [0:1:0], \\
                q_{9,1} &= p_9 \times [1:1:1], & q_{9,2} = p_9 \times [0:1:0].
            \end{align*}
         
        Letting $\pi_2: \Bl_{9,4} \P^3 \to \Bl_9 \P^3$ be the blow up of the infinitely near points, the resulting strict transform, $\overline{S}_{3,2,1}'$, of $\overline{S}_{3,2,1}$ is a smooth surface. 
        Now, let $E_i$ be the exceptional divisor from the first blow up, $F_{i,j}$, $j=1,2$, be the exceptional divisor from blowing up the infinitely near points. 
        Let $\widetilde{E}_i$ be the strict transform after blowing up the infinitely near points and $\widetilde{H}$ denote the pullback of a general hyperplane in $\P^3$. 
        By the ramification formula,
            \begin{align*}
                K_{\Bl_{9,4}\P^3} &= -4\widetilde{H} + 2\sum_{i=1}^7E_i + 2\widetilde{E}_8 + 2\widetilde{E}_9 + 4F_{8,1} + 4F_{8,2} + 4F_{9,1} + 4F_{9,2}, \\
                \overline{S}_{3,2,1}' &= 5\widetilde{H} - 2\sum_{i=1}^7 E_i - 3\widetilde{E}_8 - 3\widetilde{E}_9 - 5F_{8,1} - 5F_{8,2} - 5F_{9,1} - 5F_{9,2}.  
            \end{align*}
        Now, using the adjunction formula, we have that the canonical class of our surface is given by 
            \begin{align*}
                K_{\overline{S}_{3,2,1}'} &= (\widetilde{H} - \widetilde{E}_8 - \widetilde{E}_9 - F_{8,1} - F_{8,2} - F_{9,1} - F_{9,2})\restr{\overline{S}_{3,2,1}'} \\
                                &= \pi_2^*(\widetilde{H} - E_8 - E_9)\restr{\overline{S}_{3,2,1}'}. 
            \end{align*}
        The linear system $\abs{H - E_8 - E_9}$ corresponds to the planes that pass through the points $p_8$ and $p_9$ so that $\abs{K_{\overline{S}_{3,2,1}'}}$ are those planes intersected with $\overline{S}_{3,2,1}'$. 

        Let $L$ be the line that connects $p_8$ and $p_9$. 
        Then, $L \subset \overline{S}_{3,2,1}$. 
        A general plane intersects $\overline{S}_{3,2,1}$ along a degree five curve, so that we can express this divisor as $L'$ and $Q'$, the strict transforms of $L$ and a quartic $Q$. 
        We write $K_{\overline{S}_{3,2,1}'} = L' + Q'$. 
        By Bertini's theorem, the quartic must be irreducible. 

        The adjunction formula for $L$ shows that 
            \[
                2\mathrm{genus}(L') - 2 = (K_{\overline{S}_{3,2,1}'} + L')\restr{L'} = 2L'^2 + L'Q'. 
            \]
        Since $\mathrm{genus}(L') = 0$ and $K_{S_{3,2,1}'}L' = -1$, we have that $L'$ is a $-1$ curve on $\overline{S}_{3,2,1}'$ and $L'Q' = 0$. 
        Thus, $Q'^2 = 0$. 
        Because $L'$ is a $(-1)$-curve on $\overline{S}_{3,2,1}'$, there is a birational map contracting $L'$, $\rho: \overline{S}_{3,2,1}' \to S$. 
        This contraction then shows that $\abs{K_S}$ moves in a pencil but has $K_S^2 = 0$. 

        The linear system corresponds to a map $\abs{K_S}: S \to \P^1$, whose fibers are given by the quartics in the pencil. 
        Applying the adjunction formula once again, we see that $\mathrm{genus}(Q') = 1$. 
        The map is an elliptic fibration; the surface $S$ is an elliptic surface with Kodaira dimension $1$. 
        Thus, $\Mod{2}{1}{P}{3,2,1}$ has Kodaira dimension $1$. 
    \end{proof}


    \begin{lemma}
        The elliptic fibration $\abs{K_{\overline{S}_{3,2,1}'}}: \overline{S}_{3,2,1}' \to \P^1$ maps the hyperplane $\{CZ + DW = 0\} \subset \P^3$ to the coordinates $[C:D] \in \P^1$, where the intersection is an elliptic curve, except at finitely many $[C:D]$. 
    \end{lemma}
    \begin{proof}       
        Let $H$ be a plane in $\P^3$ given by $Ax + By + Cz + Dw = 0$, parameterized by $[A:B:C:D] \in \P^3$. 
        The planes that pass through $[1:0:0:0]$ and $[0:1:0:0]$ are given by the pencil $Cz + Dw = 0$ for $[C:D] \in \P^1$. 
        The linear system is then given by those planes intersected with $\overline{S}_{3,2,1}$. 
        Further, the intersection is given by an elliptic curve and a line connecting $[1:0:0:0]$ and $[0:1:0:0]$. 

        Recall that the $z$-coordinate of a point on the surface corresponds to the fixed point of the portrait. 
        By intersecting with these planes, the $z$-coordinate is given by $-\frac{D}{C}$, in the affine chart $w = 1$. 
        Thus, the elliptic fibration maps a tuple in $\overline{S}_{3,2,1}$ to the fixed point of the portrait, where the fibers are ellipic curves. 
        In the affine chart, $w=1$, the elliptic curve is parameterized by $(x,y) \times [C:D] \in \A^2 \times \P^1$ and is given by 
            \begin{align*}
                0 = &-x^2y^2C^3 - x^2y^2C^2D + x^2yC^3 + xy^2C^3 + x^2yC^2D + xy^2C^2D \\
                    & - xyC^3 + 3xyCD^2 + xyD^3 - xCD^2 - yCD^2 + CD^2
            \end{align*}
        A calculation shows that when $[C:D] = [1:0], [1:-1]$, or $ [2: 1 \pm i \sqrt{3}]$, the general fiber degenerates to a union of four lines, or two lines and a conic and that these are the only points where it happens in this chart. 
    \end{proof}


    \begin{lemma}
        The surface $\mathcal{M}_2^1[\mathcal{P}_{3,3}]$ has Kodaira dimension $1$. 
    \end{lemma}
    \begin{proof}
        We use the model of $\Mod{2}{1}{P}{3,3}$ that is embedded into the space of six ordered points, $P_1^6$, taking inspiration from \cite{Blanc_2015}, where the authors consider the space $\Mod{2}{1}{P}{6} \subset P_1^6 \subset \P^5$ to understand its quotient spaces. 
        Recall that $\Mod{2}{1}{P}{3,3} \subset (\P^5 \times (\P^1)^6)/\Aut{\P^1}$. 
        Thus, there is a well defined map to the space of six ordered points, $P_1^6 = (\P^1)^6/\Aut{\P^1}$. 
        
        The map from $(\P^1)^6 \dashrightarrow P_1^6 \subset \P^5$ is constructed by considering the $20$ triangles on $6$ labeled vertices. 
        There is a partition into two sets of $10$ such that two disjoint triangles are opposite colors, and every tetrahedron has $2$ triangles of each color. 
        There is a group action by the symmetric group on six elements on this set, which is the outer automorphism of the symmetric group on six elements.  
        For more details, see \cite{vakil2005space}.

        The variety $P_1^6$ can be viewed as the rational cubic threefold in $\P^5$ given by 
            \begin{align*}
                0 =& X_0 + X_1 + X_2 + X_3 + X_4 + X_5 \\
                0 =& X_0^3 + X_1^3 + X_2^3 +  X_3^3 + X_4^3 + X_5^3 \\
            \end{align*}
        Let $D_1$ be the hypersurface given by the linear condition and $D_2$ be the hypersurface given by the cubic condition. 
        Using SageMath, we see that $\overline{S}_{3,3}$, the projective closure, can be realized as a surface in in $\P^5$ given by the intersection of $D_1,D_2$, and $D_3$, where $D_3$ is given by 
            \begin{align*}             
                 0 =    & X_0^3 + 2X_1^3 + 2X_2^3 - X_1X_3^2 - X_2X_3^2 -X_3^3 + X_1^2X_4 + 2X_1X_2X_4 + X_2^2X_4 \\
                        & + 2X_1X_3X_4 + 2X_2X_3X_4 + 2X_1X_4^2  + 2X_2X_4^2 + 3X_3X_4^2  + X_4^3 - X_3^2X_5 \\
                        & + 2X_1X_4X_5 + 2X_2X_4X_5 + 2X_3X_4X_5 + 2X_4^2X_5 + X_4X_5^2 + 2X_5^3\\
            \end{align*}
        Moving forward, let $\overline{S}_{3,3}$ denote the model in the space of six ordered points. 

        The surface $\overline{S}_{3,3}$ is singular at the following thirteen points of multiplicity two
            \begin{align*}
                p_1 & = [1:1:1:-1:-1:-1], & p_2 &= [1:1:-1:1:-1:-1], & p_3 &= [1:-1:1:1:-1:-1], \\
                p_4 & = [1:-1:-1:-1:1:1], & p_5 &= [1:1:-1:-1:1:-1], & p_6 &= [1:-1:1:-1:1:-1], \\
                p_7 & = [1:-1:-1:1:-1:1], & p_8 &= [1:-1:-1:1:1:-1],  & p_9 &= [1:-1:1:-1:-1:1], \\
                p_{10} & = [1:1:-1:-1:-1:1], & p_{11} & = [0:1:-1:0:0:0:0], & p_{12} &= [0:-1:0:0:0:1] \\
                p_{13} &= [0:0:-1:0:0:1],
            \end{align*}
        and the curve $\mathcal{C} = D_1 \cap D_2 \cap D_3 \cap \{X_0 = X_3\} \cap \{X_3 = X_4 \}$. 
        Further, we have that $p_8, p_{11}, p_{12}, p_{13} \in \mathcal{C}$, where $p_8$ is a singular point of multiplicity two on $\mathcal{C}$.  

        We now blow up the singularities of $\overline{S}_{3,3}$ noting that although each hypersurface contains the points, $D_1,D_2$ vanish with multiplicity one at each singularity while $D_3$ vanishes with multiplicity two at each singularity. 
        We first resolve the curve $\mathcal{C}$ by blowing up $p_8$, and then resolving the strict transform, $\mathcal{C}'$. 
        Finally, we blow up the remaining nine singular points that do not lie on $\mathcal{C}'$.
        The points $p_{11}, p_{12}, p_{13}$ are resolved after blowing up $\mathcal{C'}$.  
        Denoting by $\pi: \widehat{\P^5} \to \P^5$ the blow up process described, the strict transform, $\overline{S}_{3,3}'$ is a smooth surface. 

        We denote by $E_q, E_{\mathcal{C}'}, E_1, ..., E_{10}$ the twelve exceptional divisors obtained according to the blow up order described above, and $\widetilde{H}$ the pull-back of a general hyperplane of $\P^5$. 
        The ramification formula gives that the canonical divisor is $K_{\widehat{\P^5}} = -6\widetilde{H} + 4E_q + 3E_{\mathcal{C}'} + 4\sum_{i=1}^{10} E_i$. 
        We see that the strict transform of each hypersurface is then equivalent to 
            \begin{align*}
                D_1' & = \widetilde{H} - E_8 - E_{\mathcal{C}'} - \sum_{i=1}^{10} E_i, \\ 
                D_2' &= 3\widetilde{H} - E_8 - E_{\mathcal{C}'} - \sum_{i=1}^{10} E_i, \\ 
                D_3' &= 3\widetilde{H} - 2E_8 - 2E_{\mathcal{C}'} - 2\sum_{i=1}^{10} E_i, 
            \end{align*}
        in $\Pic({\widehat{\P^5}})$. 
        
        Applying the adjunction formula, we find that $K_{S_{3,3}'} = \left. \left( H - E_{\mathcal{C}'} \right) \right\restr{S_{3,3}'}$. 
        The curve $\mathcal{C}$ is contained in a $\P^2 \subset \P^5$ that is defined by $D_1, \{X_0 = X_3\}, \{X_3 = X_4 \}$. 
        Using this model of $\P^2$, we see that that $\deg(\mathcal{C})= 3$, arithmetic genus is $1$, and geometric genus is $0$. 
        This is the class of a hyperplane that contains $\mathcal{C}$, however, any hyperplane that contains $\mathcal{C}$, must also contain $p_8$. 
        Since this divisor is effective, we have that $K_{\overline{S}_{3,3}'} = \pi^*(H - E_8  - E_{\mathcal{C}'})\restr{\overline{S}_{3,3}'} + E_8\restr{\overline{S}_{3,3}'}$ is effective. 
        
        Now, any hyperplane that contains $\mathcal{C}$ and $p_8$ must contain the plane defined by $D_1, \{X_0 = X_3\}, \{X_3 = X_4 \}$. 
        The only curves that this divisor can be negative on must lie on this $\P^2$. 
        We check which curves on this $\P^2$ are on $\overline{S}_{3,3}'$. 
        The intersection of $S_{3,3}$ and $\{X_0 = X_3 \}$ is the union of three lines (with multiplicity one) and the cubic curve (with multiplicity two). 
        The intersection of $S_{3,3}$ and $\{X_3 = X_4 \}$ is the union of three other lines (with multiplicity one) and the same cubic curve (with multiplicity two). 
        Let $L_1,L_2,L_3 \subset \P^5$ be these three lines. 
        If we use the hyperplane $\{X_0 = X_3 \}$, we have that 
            \begin{align*}
                L_1 \cap L_2            &= \{ p_3\},        & L_1 \cap L_3          &= \{ p_2 \},       & L_2 \cap L_3          &= \{ p_7 \}, \\
                \mathcal{C} \cap L_1    & = \{ p_{11}\},    & \mathcal{C} \cap L_2  &= \{ p_{13} \},    & \mathcal{C} \cap L_3  &= \{p_{12} \}.  
            \end{align*}
        However, none of these $L_i$ contain $p_8$, so that $L_i E_8 = 0$. 

        We can express our divisor as $K_{S_{3,3}'} = (2\widetilde{\mathcal{C}'} + L_1' + L_2' + L_3')$, where $E_8 \restr{S_{3,3}} = -2E_8^4$ and $\widetilde{\mathcal{C}'} = E_{\mathcal{C}'}\restr{\overline{S}_{3,3}'}$. 
        On $\widehat{\P^5}$, $\widetilde{\mathcal{C}'}$, and the strict transforms of the lines, $L_i'$ no longer intersect. 
        By adjunction, we have 
            \begin{align*}
                -2                      &=  K_{\overline{S}_{3,3}'} L_i' + L_i'^2                  = 2L_i'^2 \\
                0                       &= K_{\overline{S}_{3,3}'}\widetilde{\mathcal{C}'} + (\widetilde{\mathcal{C}'})^2  = 3(\widetilde{\mathcal{C}'})^2 
            \end{align*}
        Thus, $L_i'$ are $(-1)$ curves and $(\widetilde{\mathcal{C}'})^2 = 0$. 
        The canonical divisor $K_{\overline{S}_{3,3}'}$ is then nef and effective. 
        Since $(\widetilde{\mathcal{C}'})^2 = 0$, $\abs{K_{\overline{S}_{3,3}'}}: \overline{S}_{3,3}' \to \P^1$ is an elliptic fibration with $\abs{\widetilde{\mathcal{C}'}}$ as the fibers. 
        Therefore, $S_{3,3}'$ is an elliptic surface with Kodaira dimension $1$ so that $\Mod{2}{1}{P}{3,3}$ has Kodaira dimension $1$.         
    \end{proof}


    \begin{lemma}
        The surface $\Mod{2}{1}{P}{4,1,1}$ has Kodaira dimension $1$.   
    \end{lemma}
    \begin{proof}
        In the previous section we showed that $\Mod{2}{1}{P}{4,1}$ is a rational surface. 
        We have that $\Mod{2}{1}{P}{4,1,1}$ is a double cover of $\Mod{2}{1}{P}{4,1}$.
        The covering map, $p: \Mod{2}{1}{P}{4,1,1} \to \Mod{2}{1}{P}{4,1}$ extends to $S_{4,1,1} \to S_{4,1} \cong \A^2$. 
        It is branched along the discriminant of $\Phi_{1,f}^*([U:1]) = AU^2 + BU + C$, where  
            \begin{align*}
                A &= - (xy^2 - xy + x - y),  \\
                B &= -x(x^2y - 3xy^2 - x^2 + 2xy + 2y^2 - y), \\
                C &= -xy(xy^2 + x^2 - 3xy + y). 
            \end{align*}
        The branch locus, $\mathcal{B}$ is given by the union of the following irreducible curves, 
                \begin{align*}
                        L:      & 0 = x, \\
                        C_7:    &0 = x^5y^2 - 6x^4y^3 + 9x^3y^4 - 4x^2y^5 - 2x^5y + 10x^4y^2 - 12x^3y^3 + 4x^2y^4\\
                                &  + x^5 - 4x^4y + 2x^3y^2 - 2x^2y^3 + 4xy^4 - 2x^3y + 12x^2y^2 - 12xy^3 - 3xy^2 + 4y^3. 
                \end{align*} 
        Taking its projective closure in $\P^2$ with coordinates $[x:y:w]$, the curve $\overline{C}_7$ is singular at the following points,  
            \begin{align*}
                p_1 &= [0:1:0], & p_2 &= [0:0:1], &p_3 &= [1:1:1], \\ 
                p_4 &= [1:1:0], &p_5 &= [1:0:0], 
            \end{align*}
        with multiplicity three at $p_2$ and two for the rest. 
        We also have that $\overline{L}$ intersects $\overline{C}_7$ at $p_1, p_2$. 
        Letting $\pi_1: \Bl_5 \P^2 \to \P^2$ denote the blow up at the five singular points, we see that the strict transform of $\overline{C}_7$ is not smooth, it is singular at the following infinitely near double points,   
            \begin{align*}
                q_1 &= p_1 \times [0:1], & q_2 &= p_2 \times [1:0],        & q_3 &= p_3 \times [0:1], \\
                q_4 &= p_4 \times [1:-1], & q_5 &= p_5 \times [1:1].  
            \end{align*}
        The strict transform of $\overline{L}$ intersects the strict transform of $\overline{C}_7$ at $q_1$. 
        Letting $\pi_2: \Bl_{5,5} \P^2 \to \P^2$ denote the infinitely near blow up, we have that the resulting strict transform, $\overline{C}_7'$, of $\overline{C}_7$ is now smooth and does not intersect $\overline{L}'$, the final strict transform of $\overline{L}$. 
        
        We denote by $E_i$ the exceptional divisors from the blow up of $p_i$, $F_j$ the exceptional divisors from the blow up of $q_j$, and $\widetilde{H}$ the pull back of a general line in $\P^2$. 
        We express the canonical class and branch locus in $\Pic(\Bl_{5,5} \P^2)$
            \begin{align*}
                K_{\Bl_{5,5}\P^2} & = -3\widetilde{H} + \sum_{i=1}^5 E_i + 2\sum_{j=1}^5 F_j, \\ 
                \overline{\mathcal{B}}' &= 8\widetilde{H} -3E_1 - 4E_2 - 2 \sum_{i=3}^5 E_i - 5F_1 - 6F_2 - 4 \sum_{j=3}^5 F_j.  
            \end{align*}
            
        Not every coefficient of $\overline{\mathcal{B}}' \in \mathrm{Pic}(\Bl_{5,5} \P^2)$ is even. 
        Consider the modified branch locus, $\widetilde{\mathcal{B}} = \overline{\mathcal{B}}' + E_1 + F_1$ \footnote{\cite{barth2003compact} Chapter V, Section 22 If a two-fold covering is branched over a curve with a singularity of odd multiplicity, then we add in an exceptional divisor to have all even coefficients of the branch locus in the Picard group.}. 
        Using the Riemann-Hurwitz formula with $\widetilde{\mathcal{B}}$, we get 
            \begin{align*}
                K_{\overline{S}_{4,1,1}'} &= p^* \left (\widetilde{H} - E_2\right) 
            \end{align*}
        where $\overline{S}_{4,1,1}'$ denotes the resolved surface from blowing up the branch locus. 
        The linear system $\abs{\widetilde{H} - E_2}$ is corresponds to the lines in $\P^2$ that pass through $p_2$. 
        The map induced by $\abs{\widetilde{H} - E_2}$ is the projection from the the point $p_2$, that is, $\P^2 \to \P^1, [x:y:w] \mapsto [x:y]$. 
        Thus, $p^*(\widetilde{H} - E_2)$ is the pullback of those lines in $S_{4,1,1}$, hence, $\abs{p^*(\widetilde{H} - E_2)}$ defines a map to $\P^1$, so that $\Mod{2}{1}{P}{4,1,1}$ is an elliptic surface with Kodaira dimension $1$. 
    \end{proof}


    \begin{lemma}
        The surface $\Mod{2}{1}{P}{5,1}$ has Kodaira dimension $2$. 
    \end{lemma}
    \begin{proof}
        The surface $S_{5,1} \subset \A^3$ is given by the following equation
            \begin{align*}
                0 = & x^2y^2z^2 - xy^2z^3 - x^3y^2 + x^3yz - xy^3z - 2x^2yz^2 + 2xy^2z^2  \\
                    & - x^2z^3 + 2xyz^3 + x^3y^2 + xy^3 - x^3z^2 - x^2yz^2 + xy^2z^2  \\ 
                    & + 3x^2z^2 - 3xyz^2 - y^2z^2 - xy^2 + xyz^3 - xz^2 + yz^2. 
            \end{align*}
        The projective closure, $\overline{S}_{5,1}$, is singular at the following twelve points
            \begin{align*}
                p_1 &= [-1-\sqrt{5}: 3 + \sqrt{5}: z_{1,1}: 2], & p_2 &= [-1-\sqrt{5}: 3 + \sqrt{5}: z_{1,2}: 2],  \\ 
                p_3 &= [-1+\sqrt{5}: 3 - \sqrt{5}: z_{2,1}: 2], & p_4 &= [-1+\sqrt{5}: 3 - \sqrt{5}: z_{2,2}: 2], \\ 
                p_5 &= [1:0:1:1], & p_6 &= [1:1:0:1],  \\
                p_7 &= [0:1:0:1], & p_8 &= [1:1:1:1], \\
                p_9 &= [0:0:0:1], & p_{10} &= [0:0:1:0], \\ 
                p_{11} &= [0:1:0:0], &  p_{12} &= [1:0:0:0], 
            \end{align*}
        where $z_{1,i}$ is one of the solutions to $z^2 + 2(3 + \sqrt{5}) - 2z = 0$ and similarly for $z_{2,i}$. 
        The first seven $p_i$ are double points while the remaining five are triple points. 
        Let $\pi_1: \Bl_{12} \P^3 \to \P^3$ denote the blow up of the twelve points. 
        The strict transform of $\overline{S}_{5,1}$ is not smooth. 
        It is singular at the seven infinitely near double points 
            \begin{align*}
                q_8 &= p_8 \times [1:1:1], & q_9 &= p_9 \times [1:0:0], \\
                q_{10} &= p_{10} \times [0:0:1], \\
                q_{11,1} &= p_{11} \times [1:0:0], & q_{11,2} &= p_{11} \times [0:1:0],  \\ 
                q_{12,1} &= p_{12} \times [0:1:0], & q_{12,2} &= p_{12} \times [1:1:1]. 
            \end{align*}        
            
        Let $\pi_2: \Bl_{12,7} \P^3 \to \Bl_{12} \P^3$ be the blow up of the seven infinitely near points and $\overline{S}_{5,1}'$ be the strict transform, which is now smooth. 
        Let $E_i$ be the exceptional divisors from blowing up $p_1,...,p_7$, $\widetilde{E}_i$ be the pullback of the exceptional divisors from blowing up the infinitely near points, $F_i$ the exceptional divisors from blowing up $q_8, q_9, q_{10}$, $F_{i,j}$ from blowing up the remaining four $q_{i,j}$ and $\widetilde{H}$ be the pullback of a general hyperplane on $\P^3$. 
        The ramification formula gives
            \begin{align*}
                K_{\Bl_{12,7}\P^3} &= -4\widetilde{H} + 2\sum_{i=1}^7 E_i + 2\sum_{j=8}^{12}\widetilde{E}_j + 4(F_8 + F_9 + F_{10} + F_{11,1} + F_{11,2} + F_{12,1} + F_{12,2}), \\ 
                \overline{S}_{5,1}' &= 6\widetilde{H} - 2\sum_{i=1}^7 E_i - 3\sum_{j=8}^{12} \widetilde{E}_j - 5(F_8 + F_9 + F_{10} + F_{11,1} + F_{11,2} + F_{12,1} + F_{12,2}). 
            \end{align*}
        Applying the adjunction formula, we find that 
            \[
                K_{\overline{S}_{5,1}'} = \left. \left( 2\widetilde{H} - \sum_{j=8}^{12} \widetilde{E}_j - F_8 - F_9 - F_{10} - F_{11,1} - F_{11,2} - F_{12,1} - F_{12,2} \right) \right\restr{\overline{S}_{5,1}'} 
            \]
        Then, the canonical class is $\pi_2^*(2\widetilde{H} - E_8 - E_9 - E_{10} - E_{11} - E_{12})\restr{\overline{S}_{5,1}'}$, the pullback of an effective divisor. 
        Using the intersection product, $H^3 =1, E_i^3 = 1$, we have that 
            \begin{align*}
                K_{\overline{S}_{5,1}'}^2  &= \pi_2^*(2\widetilde{H} - E_8 - E_9 - E_{10} - E_{11} - E_{12})^2 \overline{S}_{5,1}' \\ 
                                &= \pi_2^*(4\widetilde{H}^2 + E_8^2 + E_9^2 + E_{10}^2 + E_{11}^2 + E_{12}^2)\overline{S}_{5,1}' \\
                                &= (4\widetilde{H}^2 + E_8^2 + E_9^2 + E_{10}^2 + E_{11}^2 + E_{12}^2)\pi_{2*}(\overline{S}_{5,1}') \\ 
                                &= (4\widetilde{H}^2 + E_8^2 + E_9^2 + E_{10}^2 + E_{11}^2 + E_{12}^2)(6\widetilde{H} - 2E_8 - 2E_9 - 2E_{10} - 2E_{11} - 2E_{12}) \\ 
                                &= 24\widetilde{H}^3 - 2E_8^3 - 2E_9^3 - 2E_{10}^3 - 2E_{11}^3 - 2E_{12}^3 \\
                                &= 14. 
            \end{align*}
        Note that $K_{\overline{S}_{5,1}}'$ is nef as it is base-point free.
        Hence, $\overline{S}_{5,1}'$ is general type and $\Mod{2}{1}{P}{5,1}$ has Kodaira dimension $2$. 
    \end{proof}


    We now have that the following is proved.
    \begin{theorem}
        Let $\mathcal{P}_{5,1}, \mathcal{P}_{4,1,1}, \mathcal{P}_{4,2}, \mathcal{P}_{3,3}, \mathcal{P}_{3,2,1}, \mathcal{P}_{3,1,1,1}$ be pre-periodic portraits on six points. 
        Then, 
            \begin{itemize}
                \item $\Mod{2}{1}{P}{3,1,1,1}$ has Kodaria dimension $-\infty$   
                \item $\Mod{2}{1}{P}{4,2}$ has Kodaira dimension $0$
                \item $\Mod{2}{1}{P}{3,2,1}, \Mod{2}{1}{P}{3,3}$, and $\Mod{2}{1}{P}{4,1,1}$ have Kodaira dimension $1$
                \item $\Mod{2}{1}{P}{5,1}$ has Kodaira dimension $2$. 
            \end{itemize}
    \end{theorem}


    \begin{corollary}
        Let $\mathcal{P}$ be a portrait. 
        Suppose that $\mathcal{P}_6$ or $ \mathcal{P}_{5,1}$ are subportraits. 
        If $\Mod{2}{1}{P}{} \neq \emptyset$, then it is a surface of general type. 
    \end{corollary}
    Using Riemann-Hurwitz, we have that the canonical divisor is still the sum of a big and effective divisor. 


    \begin{corollary}
        Suppose that the surface $\Mod{2}{1}{P}{k}$ is general type for $k \geq 6$. 
        Let $\mathcal{P}$ be a pre-periodic portrait with $N$ points where $N \geq 59$. 
        Then, the surface $\Mod{2}{1}{P}{}$ is general type. 
    \end{corollary}
    \begin{proof}
        Note that the only partitions of $N \geq 59$ that we consider have a bound on the number of $1,2,3,4,5$ that can appear. 
        Any degree two rational map on $\P^1$ has at most three fixed points, one $2$-cycle, two $3$-cycles, three $4$-cycles, and six $5$-cycles. 
        Thus, any partition of $N$ that yields a nonempty moduli space is guaranteed to have $k \geq 6$ appear in its partition. 
        Hence, there is a map, $p: \Mod{2}{1}{P}{} \to \Mod{2}{1}{P}{k}$ for $k\geq 6$. 
        Thus, the canonical divisor of $\Mod{2}{1}{P}{N}$, $p^*(K_{\Mod{2}{1}{P}{k}}) + \mathcal{R}$, is the sum of a big and effective divisor, so that the moduli space is general type. 
    \end{proof}

    \bibliographystyle{plain}
    \bibliography{References}

@article{milnor1993geometry,
    title={Geometry and dynamics of quadratic rational maps, with an appendix by the author and Lei Tan},
    author={Milnor, John},
    journal={Experimental mathematics},
    volume={2},
    number={1},
    pages={37--83},
    year={1993},
    publisher={Taylor \& Francis}
}

@misc{silverman1996spacerationalmapsp1,
    title={The space of rational maps on $\mathbb{P}^1$}, 
    author={Joseph H. Silverman},
    year={1996},
    eprint={math/9609212},
    archivePrefix={arXiv},
    primaryClass={math.DS},
    url={https://arxiv.org/abs/math/9609212}, 
}

@article{Doyle_2020,
    title={Moduli spaces for dynamical systems with portraits},
    volume={64},
    ISSN={0019-2082},
    url={http://dx.doi.org/10.1215/00192082-8642523},
    DOI={10.1215/00192082-8642523},
    number={3},
    journal={Illinois Journal of Mathematics},
    publisher={Duke University Press},
    author={Doyle, John R. and Silverman, Joseph H.},
    year={2020},
    month=sep 
}

@misc{manes2009modulispacesfamiliesrational,
    title={Moduli spaces for families of rational maps on $\mathbb{P}^1$}, 
    author={Michelle Manes},
    year={2009},
    eprint={0902.1813},
    archivePrefix={arXiv},
    primaryClass={math.NT},
    url={https://arxiv.org/abs/0902.1813}, 
}

@article{JOUANOLOU1991117,
title = {Le formalisme du résultant},
journal = {Advances in Mathematics},
volume = {90},
number = {2},
pages = {117-263},
year = {1991},
issn = {0001-8708},
doi = {https://doi.org/10.1016/0001-8708(91)90031-2},
url = {https://www.sciencedirect.com/science/article/pii/0001870891900312},
author = {J.P Jouanolou}
}

@article{Blanc_2015,
   title={Moduli Spaces of Quadratic Rational Maps with a Marked Periodic Point of Small Order},
   ISSN={1687-0247},
   url={http://dx.doi.org/10.1093/imrn/rnv063},
   DOI={10.1093/imrn/rnv063},
   journal={International Mathematics Research Notices},
   publisher={Oxford University Press (OUP)},
   author={Blanc, Jérémy and Canci, Jung Kyu and Elkies, Noam D.},
   year={2015},
   month=mar, pages={rnv063}}

@misc{vakil2005space,
      author = {Ravi Vakil},
      title = {THE SPACE OF ORDERED POINTS ON THE PROJECTIVE LINE IS ALMOST ALWAYS CUT OUT BY QUADRICS},
      year = {2005},
      month = {sept},
      note = {http://math.stanford.edu/~vakil/HMS/nptssept12.pdf},
    }

@article{Morton1994RationalPP,
  title={Rational periodic points of rational functions},
  author={Morton, Patrick and Silverman, Joseph H.},
  journal={International Mathematics Research Notices},
  volume={1994},
  number={2},
  pages={97--110},
  year={1994},
  publisher={Oxford University Press},
  doi={10.1155/S1073792894000127}
}

@misc{flynn1995cyclesquadraticpolynomialsrational,
      title={Cycles of Quadratic Polynomials and Rational Points on a Genus-Two Curve}, 
      author={E. V. Flynn and Bjorn Poonen and Edward F. Schaefer},
      year={1995},
      eprint={math/9508211},
      archivePrefix={arXiv},
      primaryClass={math.NT},
      url={https://arxiv.org/abs/math/9508211}, 
}

@misc{poonen1995completeclassificationrationalpreperiodic,
      title={The Complete Classification of Rational Preperiodic Points of Quadratic Polynomials over Q: A Refined Conjecture}, 
      author={Bjorn Poonen},
      year={1995},
      eprint={math/9512217},
      archivePrefix={arXiv},
      primaryClass={math.NT},
      url={https://arxiv.org/abs/math/9512217}, 
}

@book{barth2003compact,
  title={Compact Complex Surfaces},
  author={Barth, W. and Hulek, K. and Peters, C. and van de Ven, A.},
  isbn={9783540008323},
  lccn={2003063511},
  series={Ergebnisse der Mathematik und ihrer Grenzgebiete. 3. Folge / A Series of Modern Surveys in Mathematics},
  url={https://books.google.com/books?id=LtWDVZxiK6EC},
  year={2003},
  publisher={Springer Berlin Heidelberg}
}

@manual{sagemath,
  Key          = {SageMath},
  Author       = {{The Sage Developers}},
  Title        = {{S}ageMath, the {S}age {M}athematics {S}oftware {S}ystem ({V}ersion 9.5)},
  note         = {{\tt https://www.sagemath.org}},
  Year         = {2022},
}

    \section{Acknowledgements}
        I am grateful to my advisor, John Lesieutre, for his useful and invaluable guidance and helpful comments throughout this work. I would like to thank several graduate students at Penn State (Jacob Canel, Austin Davis, Andy B. Day, Eugene Henninger-Voss, Qitong Jiang, Neelarnab Raha, Pisya Vikash, and Xingkai Wang).
        The \texttt{SageMath} system was vital in carrying out many of the computations. \phantom{\cite{sagemath}}

\end{document}